\documentclass[a4paper,numbook,babel,final,envcountsame,11pt]{article}

\usepackage{amssymb}
\usepackage{amsthm}
\usepackage{amsmath}
\usepackage{graphicx}
\usepackage{array,hhline}

\newtheorem{lemma}{Lemma}[section]
\newtheorem{theorem}[lemma]{Theorem}
\newtheorem{proposition}[lemma]{Proposition}
\newtheorem{conjecture}[lemma]{Conjecture}
\newtheorem{corollary}[lemma]{Corollary}
\theoremstyle{definition}
\newtheorem{definition}[lemma]{Definition}
\newtheorem{remark}[lemma]{Remark}

\numberwithin{equation}{section}
\numberwithin{figure}{section}

\newcommand{\Cset}{\mathcal{C}}

\newcommand{\Lset}{\mathcal{L}}

\newcommand{\Xset}{\mathcal{X}}
\newcommand{\Yset}{\mathcal{Y}}

\begin{document}

\title{\huge On the usage of lines in $GC_n$ sets}         
\author{Hakop Hakopian, Vahagn Vardanyan}        
\date{}          

\maketitle

\begin{abstract}
A planar node set $\mathcal X,$ with $|\mathcal X|=\binom{n+2}{2}$ is called  $GC_n$ set if each node possesses fundamental polynomial in form of a product of $n$ linear factors. We say that a node uses a line $Ax+By+C=0$ if $Ax+By+C$ divides the fundamental polynomial of the node.  A line is called $k$-node line if it passes through
exactly $k$-nodes of $\mathcal X.$ At most $n+1$ nodes can be collinear in $GC_n$ sets and an $(n+1)$-node line is called maximal line. The Gasca - Maeztu conjecture (1982) states that every $GC_n$ set has a maximal line. Until now the conjecture has been
proved only for the cases $n \le 5.$
Here we adjust and prove a conjecture proposed in the paper - V. Bayramyan, H. H.,
Adv Comput Math,  43: 607-626, 2017.
Namely, by assuming that the Gasca-Maeztu conjecture is true, we prove that for any $GC_n$ set $\mathcal X$ and any $k$-node line $\ell$ the following statement holds:

\noindent Either the line $\ell$ is not used at all, or it is used
by exactly $\binom{s}{2}$ nodes of $\mathcal X,$ where $s$ satisfies the condition $\sigma:=2k-n-1\le s\le k.$
If in addition $\sigma \ge 3$ and $\mu(\mathcal X)>3$ then the first case here is excluded, i.e., the line $\ell$ is necessarily a used line. Here $\mu(\mathcal X)$ denotes the number of maximal lines of $\mathcal X.$

At the end, we bring a characterization for the usage of $k$-node lines in $GC_n$ sets when $\sigma=2$ and $\mu(\mathcal X)>3.$
\end{abstract}



\section{Introduction\label{sec:intro}}
An $n$-poised set $\mathcal X$ in the plane is a node set for which the interpolation problem with bivariate polynomials of total degree at most $n$ is unisolvent.
Node sets with geometric characterization: $GC_n$ sets, introduced by Chang and Yao \cite{CY77}, form an important subclass of $n$-poised sets. In a $GC_n$ set the fundamental polynomial of each node is a product of n linear factors. We say  that a node uses a line if the line is a factor of the fundamental polynomial of this node. A line is called $k$-node line if it passes through
exactly $k$-nodes of $\mathcal X.$ It is a simple fact that at most $n+1$ nodes can be collinear in $GC_n$ sets. An $(n+1)$-node line is called a maximal line. The conjecture of M. Gasca and J. I. Maeztu \cite{GM82} states that every $GC_n$ set has a maximal line. Until now the conjecture has been
proved only for the cases $n \le 5$ (see \cite{B90} and \cite{HJZ14}).  For a maximal line $\lambda$ in a $GC_n$ set $\mathcal X$ the following statement is evident:
the line $\lambda$ is used by all $\binom{n+1}{2}$ nodes in $\mathcal
X\setminus \lambda.$ This immediately follows from the fact that if a polynomial of total degree at most $n$ vanishes at $n+1$ points of a line then the line divides the polynomial (see Proposition \ref{prp:n+1points}, below).
Here we consider a conjecture proposed in the paper \cite{BH} by V. Bayramyan and H. H.,
concerning the usage of any $k$-node line of $GC_n$ set.
In this paper we make a correction in the mentioned conjecture and then prove it.
Namely, by assuming that the Gasca-Maeztu conjecture is true, we prove that for any $GC_n$ set $\mathcal X$ and any $k$-node line $\ell$ the following statement holds:

\noindent The line $\ell$ is not used at all, or it is used
by exactly $\binom{s}{2}$ nodes of $\mathcal X,$ where $s$ satisfies the condition $\sigma=\sigma(\Xset,\ell):=2k-n-1\le s\le k.$
If in addition $\sigma \ge 3$ and $\mu(\mathcal X)>3$ then the first case here is excluded, i.e., the line $\ell$ is necessarily a used line. Here $\mu(\mathcal X)$ denotes the number of maximal lines of $\mathcal X.$
We prove also that the subset of nodes of $\mathcal X$ that use the line $\ell$ forms a $GC_{s-2}$ set if it is not an empty set. Moreover we prove that actually it is a $\ell$-proper subset of $\mathcal X,$ meaning that it can be obtained from $\mathcal X$ by subtracting the nodes in subsequent maximal lines, which do not intersect the line $\ell$ at a node of $\Xset$ or the nodes in pairs of maximal lines intersecting $\ell$ at the same node of $\Xset.$ At the last step, when the line $\ell$ becomes maximal, the nodes in $\ell$ are subtracted (see the forthcoming Definition \ref{def:proper}).

At the end, we bring a characterization for the usage of $k$-node lines in $GC_n$ sets when $\sigma=2$ and $\mu(\mathcal X)>3.$

Let us mention that earlier Carnicer and Gasca proved that a $k$-node line $\ell$ can be used by atmost $\binom{k}{2}$ nodes of a $GC_n$ set $\mathcal X$ and in addition there are no $k$ collinear nodes that use $\ell$, provided that GM conjecture is true (see \cite{CG03}, Theorem 4.5).

\subsection{Poised sets}

Denote by $\Pi_n$ the space of bivariate polynomials of total degree
at most $n:$
\begin{equation*}
\Pi_n=\left\{\sum_{i+j\leq{n}}c_{ij}x^iy^j
\right\}.
\end{equation*}
We have that
\begin{equation} \label{N=}
N:=\dim \Pi_n=\binom{n+2}{2}.
\end{equation}
Let $\Xset$ be a set of $s$ distinct nodes (points):
\begin{equation*}
{\mathcal X}={\mathcal X}_s=\{ (x_1, y_1), (x_2, y_2), \dots , (x_s, y_s) \} .
\end{equation*}
The Lagrange bivariate interpolation problem is: for given set of values
$\Cset_s:=\{ c_1, c_2, \dots , c_s \}$ find a polynomial $p \in \Pi_n$ satisfying
the conditions
\begin{equation}\label{int cond}
p(x_i, y_i) = c_i, \ \ \quad i = 1, 2, \dots s.
\end{equation}

\begin{definition}
A set of nodes ${\mathcal X}_s$ is called \emph{$n$-poised} if
for any set of values $\Cset_s$ there exists a unique polynomial
$p \in \Pi_n$ satisfying the conditions \eqref{int cond}.
\end{definition}
It is an elementary Linear Algebra fact that if a node set $\Xset_s$ is $n$-poised then $s=N.$ Thus from now on we will consider sets $\Xset=\Xset_N$ when
$n$-poisedness is studied. If a set $\Xset$ is $n$-poised then we say that $n$ is the degree of the set $\Xset.$

\begin{proposition} \label{prp:poised}
The set of nodes ${\mathcal X}_N$ is $n$-poised if and only if the
following implication holds:
$$p \in \Pi_n,\ p(x_i,
y_i) = 0, \quad i = 1, \dots , N \Rightarrow p = 0.$$
\end{proposition}

A polynomial $p \in \Pi_n$ is called an \emph{$n$-fundamental
polynomial} for a node $ A = (x_k, y_k) \in {\mathcal X}_s,\ 1\le k\le s$,  if
\begin{equation*}
p(x_i, y_i) = \delta _{i k}, \  i = 1, \dots , s ,
\end{equation*}
where $\delta$ is the Kronecker symbol.

Let us denote the
$n$-fundamental polynomial of the node $A \in{\mathcal X}_s$ by $p_A^\star=p_{A, {\mathcal X}}^\star.$

A
polynomial vanishing at all nodes but one is also called fundamental, since it is a nonzero
constant times the fundamental polynomial.

\begin{definition}
Given an $n$-poised set $ {\mathcal X}.$ We say that a node
$A\in{\mathcal X}$ \emph{uses a line $\ell\in \Pi_1$,} if
\begin{equation*}
  p_A^\star = \ell q, \ \text{where} \ q\in\Pi_{n-1}.
\end{equation*}
\end{definition}
The following proposition is well-known (see, e.g., \cite{HJZ09b}
Proposition 1.3):
\begin{proposition}\label{prp:n+1points}
Suppose that a polynomial $p \in
\Pi_n$ vanishes at $n+1$ points of a line $\ell.$ Then we have that
\begin{equation*}
p = \ell r, \ \text{where} \ r\in\Pi_{n-1}.
\end{equation*}
\end{proposition}

\noindent Thus at most $n+1$ nodes of an $n$-poised set $\Xset$ can be collinear.
A line $\lambda$ passing through $n+1$
nodes of the set ${\mathcal X}$ is called a \emph{maximal line}.
Clearly, in view of Proposition \ref{prp:n+1points}, any maximal line $\lambda$ is used by all the nodes in $\Xset\setminus \lambda.$

Below we bring other properties of maximal lines:
\begin{corollary}[\cite{CG00}, Prop. 2.1]\label{properties}
Let ${\mathcal X}$ be an $n$-poised set. Then we have that
\vspace{-1.5mm} \begin{enumerate} \setlength{\itemsep}{0mm}
\item
Any two maximal lines of $\Xset$ intersect necessarily at a node of $\Xset;$
\item
Any three maximal lines of $\Xset$ cannot be concurrent;
\item
$\Xset$ can have at most $n+2$ maximal lines.
\end{enumerate}
\end{corollary}


\section{$GC_n$ sets and the Gasca-Maeztu conjecture \label{ss:GMconj}}

Now let us consider a special type of $n$-poised sets satisfying a geometric characterization (GC) property introduced by K.C. Chung and T.H. Yao:
\begin{definition}[\cite{CY77}]
An n-poised set ${\mathcal X}$ is called \emph{$GC_n$ set} (or $GC$ set)
 if  the
$n$-fundamental polynomial of each node $A\in{\mathcal X}$ is a
product of $n$  linear factors.
\end{definition}
\noindent Thus, $GC_n$ sets are the sets each node of
which uses exactly $n$ lines.

\begin{corollary}[\cite{CG03}, Prop. 2.3] \label{crl:minusmax}
Let $\lambda$ be a maximal line of  a $GC_n$ set ${\mathcal X}.$  Then the set  ${\mathcal X}\setminus \lambda$ is a $GC_{n-1}$ set.
Moreover, for any node $A\in \mathcal X\setminus \lambda$ we have that
\begin{equation}\label{aaaaa}p_{A, {\mathcal X}}^\star= \lambda p_{A, {\{\mathcal X\setminus \lambda}\}}^\star.\end{equation}
\end{corollary}

Next we present the Gasca-Maeztu conjecture, briefly called GM conjecture:

\begin{conjecture}[\cite{GM82}, Sect. 5]\label{conj:GM}
Any $GC_n$ set possesses a maximal line.
\end{conjecture}

\noindent  Till now, this conjecture has been confirmed for the degrees
$n\leq 5$ (see \cite{B90}, \cite{HJZ14}).
For a generalization of the Gasca-Maeztu conjecture to maximal curves see \cite{HR}.

In the sequel we will make use of the following important result:

\begin{theorem}[\cite{CG03}, Thm. 4.1]\label{thm:CG}
If the Gasca-Maeztu conjecture is true for all $k\leq n$, then any $GC_n$ set possesses at least
three maximal lines.
\end{theorem}

This yields, in view of Corollary \ref{properties} (ii) and Proposition \ref{prp:n+1points}, that each node of a $GC_n$  set $\Xset$
uses at least one maximal line.

Denote by  $\mu:=\mu(\Xset)$ the number of maximal lines of the node set  $\Xset.$

\begin{proposition}[\cite{CG03}, Crl. 3.5]\label{prp:CG-1} Let $\lambda$ be a maximal line of a $GC_n$ set $\Xset$ such that $\mu(\Xset\setminus \lambda)\ge 3.$ Then we have that $$\mu(\Xset\setminus \lambda)=\mu(\Xset)\quad \hbox{or}\quad \mu(\Xset)-1.$$
\end{proposition}

\begin{definition}[\cite{CG01}]  Given an $n$-poised set $\Xset$ and a  line $\ell.$ Then
$\Xset_\ell$ is the subset of nodes of
$\Xset$ which use the line $\ell.$
\end{definition}
Note that a statement on maximal lines we have mentioned already can be expressed as follows
\begin{equation}\label{maxaaa} \Xset_\ell=\Xset\setminus \ell, \ \hbox{if $\ell$ is a maximal line}.
\end{equation}

Suppose that $\lambda$ is a maximal line of $\Xset$ and $\ell\neq \lambda$ is any line. Then in view of the relation \eqref{aaaaa} we have that
\begin{equation}\label{rep}
\Xset_\ell\setminus \lambda=(\Xset\setminus \lambda)_\ell.
\end{equation}

In the sequel we will use frequently the following two lemmas of Carnicer and Gasca.

Let $\Xset$ be an $n$-poised set and ${\ell}$ be a line with $|\ell\cap\Xset|\le n.$
A maximal line $\lambda$ is called $\ell$-\emph{disjoint} if
\begin{equation}\label{aaaa} \lambda \cap {\ell} \cap  \Xset =\emptyset.
\end{equation}

\begin{lemma}[\cite{CG03}, Lemma 4.4]\label{lem:CG1}
Let $\Xset$ be an $n$-poised set and ${\ell}$ be a line with $|\ell\cap\Xset|\le n.$ Suppose also
that a maximal line $\lambda$ is $\ell$-disjoint.
Then we have that
\begin{equation}\label{aaa}
\Xset_{\ell} = {(\Xset \setminus \lambda)}_{\ell}.
\end{equation}
Moreover, if ${\ell}$ is an $n$-node line then we have that
$\Xset_{\ell} = \Xset \setminus (\lambda \cup {\ell}),$ hence $\Xset_{\ell}$ is an $(n-2)$-poised set.
\end{lemma}
Let $\Xset$ be an $n$-poised set and ${\ell}$ be a line with $|\ell\cap\Xset|\le n.$
Two maximal lines $\lambda', \lambda''$ are called $\ell$-\emph{adjacent} if
\begin{equation}\label{bbbb}  \lambda' \cap \lambda''\cap {\ell} \in \Xset.\end{equation}

\begin{lemma}[\cite{CG03}, proof of Thm. 4.5]\label{lem:CG2}
Let $\Xset$ be an $n$-poised set and ${\ell}$ be a line with $3\le|\ell\cap\Xset|\le n.$ Suppose also
that two maximal lines $\lambda', \lambda''$ are $\ell$-adjacent.
Then we have that
\begin{equation}\label{bbb}
\Xset_{\ell} = {(\Xset \setminus (\lambda' \cup \lambda''))}_{\ell}.
\end{equation}
Moreover, if ${\ell}$ is an $n$-node line then we have that
$\Xset_{\ell} = \Xset \setminus (\lambda' \cup \lambda'' \cup {\ell}),$ hence $\Xset_\ell$ is an $(n-3)$-poised set.
\end{lemma}

\noindent Next, by the motivation of above two lemmas, let us introduce the concept of an $\ell$-reduction of a $GC_n$ set.
\begin{definition}\label{def:reduct} Let $\Xset$ be a $GC_n$ set, $\ell$ be a $k$-node line $k\ge 2.$ We say that a set $\Yset\subset\Xset$ is an $\ell$-reduction of $\Xset,$ and briefly denote this by $\Xset\searrow_\ell\Yset,$ if
$$\Yset=\Xset \setminus \left(\Cset_0\cup \Cset_1\cup \cdots \cup\Cset_k\right),$$
where
\begin{enumerate}
\item
$\Cset_0$ is an $\ell$-disjoint  maximal line of $\Xset,$ or $\Cset_0$ is the union of a pair of $\ell$-adjacent maximal lines of $\Xset;$
\item
$\Cset_i$ is an $\ell$-disjoint maximal line of the $GC$ set $\Yset_i:=\Xset\setminus (\Cset_0\cup \Cset_1\cup \cdots \cup \Cset_{i-1}),$ or $\Cset_i$  is the union of a pair of $\ell$-adjacent maximal lines of $\Yset_i,\ i=1,\ldots k;$
\item
$\ell$ passes through at least $2$ nodes of $\Yset.$
\end{enumerate}
\end{definition}

Note that, in view of Corollary \ref{crl:minusmax}, the set $\Yset$ here is a $GC_m$ set, where
$m=n-\sum_{i=0}^{k}\delta_i,$ and $\delta_i =1$ or $2$ if
$\Cset_i$ is an
 $\ell$-disjoint maximal line or a union of a pair of $\ell$-adjacent maximal lines, respectively.

We get immediately from Lemmas \ref{lem:CG1} and \ref{lem:CG2} that
\begin{equation}\label{abcd}\Xset\searrow_\ell\Yset \Rightarrow \Xset_\ell=\Yset_\ell.\end{equation}
Notice that we cannot do any further $\ell$-reduction with the set $\Yset$ if the line $\ell$ is a maximal line here.
For this situation we have the following
\begin{definition}\label{def:proper} Let $\Xset$ be a $GC_n$ set, $\ell$ be a $k$-node line, $k\ge 2.$ We say that the set $\Xset_\ell$  is $\ell$-proper $GC_m$ subset of $\Xset$
if there is a $GC_{m+1}$ set $\Yset$ such that
\begin{enumerate}
\item
$\Xset\searrow_\ell\Yset;$
\item
The line $\ell$ is a maximal line in $\Yset.$
\end{enumerate}
\end{definition}
Note that, in view of the relations \eqref{abcd} and \eqref{maxaaa} here we have that
$$\Xset_\ell = \Yset\setminus \ell=\Xset \setminus \left(\Cset_0\cup \Cset_1\cup \cdots \cup\Cset_k\cup\ell\right),$$
where the sets $\Cset_i$ satisfy conditions listed in Definition \ref{def:reduct}.

In view of this relation and Corollary \ref{crl:minusmax} we get that infact $\Xset_\ell$  is a $GC_m$ set, if it is an $\ell$-proper $GC_m$ subset of $\Xset.$

Note also that the node set $\Xset_\ell$ in Lemma \ref{lem:CG1} or in Lemma \ref{lem:CG2} is an $\ell$-proper subset of $\Xset$ if $\ell$ is an $n$-node line.

We immediately get from Definitions \ref{def:reduct} and \ref{def:proper} the following
\begin{proposition} \label{proper} Suppose that $\Xset$ is a $GC_n$ set.
 If $\Xset\searrow_\ell\Yset$ and $\Yset_\ell$ is a proper $GC_m$ subset of $\Yset$ then $\Xset_\ell$ is a proper $GC_m$ subset of $\Xset.$
\end{proposition}

\subsection{Classification of $GC_n$ sets\label{ssec:SE}}

Here we will consider  the results of Carnicer, Gasca, and God\'es concerning the classification of $GC_n$ sets according to the number of maximal lines the sets possesses. Let us start with

\begin{theorem}[\cite{CGo10}]\label{th:CGo10} Let $\Xset$ be a $GC_n$ set with $\mu(\Xset)$ maximal lines. Suppose also that $GM$ conjecture is true for the degrees not exceeding $n.$ Then $\mu(\Xset)\in\left\{3, n-1, n,n+1,n+2\right\}.$
\end{theorem}



\noindent \emph{1. Lattices with $n+2$ maximal lines - the Chung-Yao natural lattices.}

\vspace{.2cm}

Let a set $\Lset$ of $n+2$ lines be in general position, i.e., no two lines are parallel and no three lines are concurrent, $n\ge 1.$ Then the Chung-Yao set is defined as
the set $\Xset$ of all $\binom{n+2}{2}$ intersection points of these lines. Notice that the $n+2$ lines of $\Lset$ are maximal for $\Xset.$ Each fixed node
here is lying in exactly $2$ lines and does not belong to the
remaining $n$ lines. Observe that the product of the latter $n$ lines gives
the fundamental polynomial of the fixed node. Thus $\Xset$ is a
$GC_n$ set.

Let us mention that any $n$-poised set with $n+2$ maximal lines clearly forms a Chung-Yao lattice. Recall that there are no $n$-poised sets with more maximal lines  (Proposition \ref{properties}, (iii)).
\vspace{.2cm}

\noindent \emph{2. Lattices with $n+1$ maximal lines - the Carnicer-Gasca lattices.}
\vspace{.2cm}

Let a set $\Lset$ of $n+1$ lines be in general position, $n\ge 2.$ Then
the Carnicer-Gasca lattice $\Xset$ is defined as $\Xset:=\Xset^\times\cup\Xset',$
where $\Xset^\times$ is the set of all intersection points of these $n+1$ lines and $\Xset'$ is a set of other $n+1$  non-collinear nodes, called "free" nodes,
one in each line, to make the line maximal.
We have that $|\Xset|=\binom{n+1}{2}+(n+1)=\binom{n+2}{2}.$ Each fixed "free" node
here is lying in exactly $1$ line. Observe, that the product of the remaining $n$ lines gives
the fundamental polynomial of the fixed "free" node. Next, each fixed intersection node
is lying in exactly $2$ lines. The product of the remaining $n-1$ lines and the line passing through the two "free" nodes in the $2$ lines gives
the fundamental polynomial of the fixed intersection node. Thus $\Xset$ is a
$GC_n$ set. It is easily seen that $\Xset$ has exactly $n+1$ maximal lines, i.e., the lines of $\Lset.$

Let us mention that any $n$-poised set with exactly $n+1$ maximal lines clearly forms a Carnicer-Gasca lattice (see \cite{CG00}, Proposition 2.4).
\vspace{.2cm}

\noindent \emph{3. Lattices with $n$ maximal lines.}
\vspace{.2cm}

Suppose that a $GC_n$ set $\Xset$ possesses exactly $n$ maximal lines, $n\ge 3.$ These lines are in a general position and we have that $\binom{n}{2}$ nodes of $\Xset$ are intersection nodes of these lines. Clearly in each maximal line there are $n-1$ intersection nodes and therefore there are $2$ more nodes, called "free" nodes, to make the line maximal. Thus
$N-1=\binom{n}{2}+2n$ nodes are identified. The last node $O,$ called outside node, is outside the maximal lines.

Thus the lattice $\Xset$ has the following construction
\begin{equation}\label{01O}\Xset:=\Xset^\times\cup\Xset''\cup \{O\},\end{equation}
where $\Xset^\times$ is the set of intersection nodes, and $\Xset''$ is the set of $2n$  "free" nodes.

In the sequel we will need the following characterization of $GC_n$ sets with exactly
$n$ maximal lines due to Carnicer and Gasca:

\begin{proposition}[\cite{CG00}, Prop. 2.5]\label{prp:nmax}
A set $\Xset$ is a $GC_n$ set with exactly $n,\ n\ge 3,$ maximal lines $\lambda_1,\ldots,\lambda_n,$ if and only if the representation \eqref{01O} holds with the following additional properties:
\vspace{-1.5mm} \begin{enumerate} \setlength{\itemsep}{0mm}
\item  There are $3$ lines $\ell_1^o, \ell_2^o,\ell_3^o$ concurrent at the outside node $O: \ O=\ell_1^o\cap \ell_2 ^o\cap \ell_3^o$ such  that  $\Xset''\subset \ell_1^o\cup \ell_2^o \cup \ell_3^o;$
\item No line $\ell_i^o,\ i=1,2,3,$ contains $n+1$ nodes of $\Xset.$
\end{enumerate}
\end{proposition}
\vspace{.2cm}

\noindent \emph{4. Lattices with $n-1$ maximal lines.}
\vspace{.2cm}

Suppose that a $GC_n$ set $\Xset$ possesses exactly $n-1$ maximal lines: $\lambda_1,\ldots,\lambda_{n-1},$ where $n\ge 4.$ These lines are in a general position and we have that $\binom{n-1}{2}$ nodes of $\Xset$ are intersection nodes of these lines. Now, clearly in each maximal line there are $n-2$ intersection nodes and therefore there are $3$ more nodes, called "free" nodes, to make the line maximal. Thus
$N-3=\binom{n-1}{2}+3(n-1)$ nodes are identified. The last $3$ nodes $O_1,O_2,O_3$ called outside nodes, are outside the maximal lines. Clearly, the outside nodes are non-collinear.
Indeed, otherwise the set $\Xset$ is lying in $n$ lines, i.e., $n-1$ maximal lines and the line passing through the outside nodes. This, in view of Proposition \ref{prp:poised}, contradicts the $n$-poisedness of $\Xset.$

Thus the lattice $\Xset$ has the following construction

\begin{equation}\label{01OOO} \Xset:=\Xset^\times\cup\Xset'''\cup \{O_1,O_2,O_3\},\end{equation}
where $\Xset^\times$  is the set of intersection nodes, and $\\$ $\Xset''' =\left\{ A_{i}^1, A_i^2,A_i^3\in \lambda_i,  : 1\le i\le n-1\right\},$
is the set of $3(n-1)$ "free" nodes.

Denote by $\ell_{i}^{oo},\ 1\le i\le 3,$ the line passing through the two outside nodes $\{O_1,O_2,O_3\}\setminus \{O_i\}.$ We call this lines $OO$ lines.

In the sequel we will need the following characterization of $GC_n$ sets with exactly
$n-1$ maximal lines due to Carnicer and God\'es:

\begin{proposition}[\cite{CGo07}, Thm. 3.2]\label{prp:n-1max}
A set $\Xset$ is a $GC_n$ set with exactly $n-1$ maximal lines $\lambda_1,\ldots,\lambda_{n-1},$ where $n\ge 4,$ if and only if, with some permutation of the indexes of the maximal lines and "free" nodes, the representation \eqref{01OOO} holds with the following additional properties:
\vspace{-1.5mm} \begin{enumerate} \setlength{\itemsep}{0mm}
\item  $\Xset'''\setminus \{ A_{1}^1, A_2^2,A_3^3\}\subset \ell_{1}^{oo}\cup \ell_{2}^{oo} \cup \ell_{3}^{oo};$
\item Each line $\ell_{i}^{oo}, i=1,2,3,$ passes through exactly $n-2$ "free" nodes (and through $2$ outside nodes). Moreover, $\ell_{i}^{oo}\cap\lambda_i\notin\Xset,\ i=1,2,3;$
\item  The triples $\{O_1,A_2^2,A_3^3\},\ \{O_2,A_1^1,A_3^3\},\ \{O_3,A_1^1,A_2^2\}$ are collinear.
\end{enumerate}
\end{proposition}
\vspace{.2cm}

\noindent \emph{5. Lattices with $3$ maximal lines - generalized principal lattices.}
\vspace{.2cm}

A principal lattice is defined as an affine image of the set
$$PL_n:=\left\{(i,j)\in{\mathbb N}_0^2 : \quad i+j\le n\right\}.$$
Observe that  the following $3$ set of $n+1$ lines, namely $\{x=i:\ i=0,1,\ldots,n+1\},\ \  \{y=j:\ j=0,1,\ldots,n+1\}$ and $\{x+y=k:\ k=0,1,\ldots,n+1\},$ intersect at $PL_n.$
We have that $PL_n$ is a $GC_n$ set. Moreover, clearly the following polynomial is the fundamental polynomial of the node $(i_0,j_0)\in PL_n:$
\begin{equation}\label{aaabbc} p_{i_0j_0}^\star(x,y)  =\prod_{0\le i<i_0,\ 0\le j<j_0,\ 0\le k<k_0} (x-i)(y-j)(x+y-n+k),\end{equation}
where $k_0=n-i_0-j_0.$

Next let us bring the definition of the generalized principal lattice due to Carnicer, Gasca and God\'es (see \cite{CG05}, \cite{CGo06}):
\begin{definition}[\cite{CGo06}] \label{def:GPL} A node set $\Xset$ is called a generalized principal lattice, briefly $GPL_n$ if there are $3$ sets of lines each containing $n+1$ lines   
\begin{equation}\label{aaagpl}\ell_i^j(\Xset)_{i\in \{0,1,\ldots,n\}}, \ j=0,1,2,\end{equation}
such that the $3n+3$ lines are distinct,
$$\ell_i^0(\Xset)\cap \ell_j^1(\Xset) \cap \ell_k^2(\Xset) \cap \Xset \neq \emptyset \iff i+j+k=n$$
and
$$\Xset=\left\{x_{ijk}\ |\ x_{ijk}:=\ell_i^0(\Xset)\cap \ell_j^1(\Xset) \cap \ell_k^2(\Xset),
0\le i,j,k\le n, i+j+k=n\right\}.$$
\end{definition}

Observe that if $0\le i,j,k\le n,\ i+j+k=n$ then the three lines $\ell_i^0(\Xset), \ell_j^1(\Xset),  \ell_k^2(\Xset)$ intersect at a node $x_{ijk}\in \Xset.$ This implies that a node of $\Xset$ belongs to only one line of each of the three sets of $n+1$ lines. Therefore $|\Xset|=(n+1)(n+2)/2.$

One can find readily, as in the case of $PL_n$, the fundamental polynomial of each node $x_{ijk}\in \Xset,\ i+j+k=n:$
\begin{equation}\label{aaabbc} p_{i_0j_0k_0}^\star  =\prod_{0\le i<i_0,\ 0\le j<j_0,\ 0\le k<k_0}\ell_{i}^0(\Xset) \ell_{j}^1(\Xset)  \ell_{k}^2(\Xset).\end{equation}
Thus $\Xset$ is a $GC_n$ set.

Now let us bring a characterization for $GPL_n$ set due to Carnicer and God\'es:

\begin{theorem}[\cite{CGo06}, Thm. 3.6]\label{th:CGo06} Assume that GM Conjecture holds for all degrees up to $n-3$. Then the following statements are equivalent:
\vspace{-1.5mm} \begin{enumerate} \setlength{\itemsep}{0mm}
\item
$\Xset$ is generalized principal lattice of degree n;
\item
$\Xset$ is a $GC_n$ set with exactly $3$ maximal lines.
\end{enumerate}
\end{theorem}

\section{A conjecture concerning $GC_n$ sets \label{s:conj}}

Now we are in a position to formulate and prove the corrected version of the conjecture proposed in the paper \cite{BH} of V. Bayramyan and H. H.:

\begin{conjecture}[\cite{BH}, Conj. 3.7]\label{mainc}
Assume that GM Conjecture holds for all degrees up to $n$. Let $\Xset$ be a $GC_n$ set, and
${\ell}$ be a $k$-node line, $2\le k\le n+1.$ Then we have that
\begin{equation} \label{1bbaa} \Xset_{\ell} =\emptyset,\ \hbox{or} \ \Xset_\ell \ \hbox{is an $\ell$-proper $GC_{s-2}$  subset  of $\Xset,$ hence}\  |\Xset_{\ell}| = \binom{s}{2},
\end{equation}
where $\sigma:=2k-n-1\le s \le k.$\\
Moreover, if $\sigma\ge 3$ and $\mu(\Xset) >3$ then we have that $\Xset_{\ell}\neq \emptyset.$\\
Furthermore, for any maximal line $\lambda$ we have:
$|\lambda\cap \Xset_{\ell}| = 0$ or $|\lambda\cap \Xset_{\ell}| = s-1.$
\end{conjecture}

In the last subsection we characterize constructions of $GC_n$ sets for which there are non-used $k$-node lines with $\sigma=2$ and $\mu(\Xset)>3.$

Let us mention that in the original conjecture in \cite{BH} the possibility that the set
$\Xset_{\ell}$ may be empty in \eqref{1bbaa} was not foreseen. Also we added here the statement that $\Xset_\ell$ is an $\ell$-proper $GC$ subset.

\subsection{Some known special cases of Conjecture \ref{mainc} \label{s:conj}}

The following theorem concerns  the special case $k=n$ of Conjecture \ref{mainc}. It is a corrected version of the original result in \cite{BH}: Theorem 3.3. This result was the first step toward the Conjecture \ref{mainc}.
 The corrected version appears in \cite{HV}, Theorem 3.1.
\begin{theorem}\label{thm:corrected} \label{th:corrected}
Assume that GM Conjecture holds for all degrees up to
$n$. Let $\Xset$ be a $GC_n$ set, $n \ge 1,\ n\neq 3,$ and ${\ell}$ be an $n$-node line. Then we have
that
\begin{equation} \label{2bin} |\Xset_{\ell}| = \binom{n}{2}\quad \hbox{or} \quad \binom{n-1}{2}.
\end{equation}
Moreover, the following hold:
\vspace{-1.5mm} \begin{enumerate} \setlength{\itemsep}{0mm}
\item
$|\Xset_{\ell}| = \binom{n}{2}$ if and only if there is an $\ell$-disjoint maximal line $\lambda,$ i.e., $\lambda \cap {\ell} \cap  \Xset =\emptyset.$ In this case
we have that $\Xset_{\ell}=\Xset\setminus (\lambda\cup \ell).$ Hence it is an $\ell$-proper $GC_{n-2}$ set;

\item
$|\Xset_{\ell}| = \binom{n-1}{2}$ if and only if there is a pair of $\ell$-adjacent maximal lines $\lambda', \lambda'',$ i.e., $\lambda' \cap \lambda'' \cap {\ell} \in \Xset.$
In this case we have that $\Xset_{\ell}=\Xset\setminus (\lambda'\cup \lambda''\cup \ell).$ Hence it is an $\ell$-proper $GC_{n-3}$ set.
\end{enumerate}
\end{theorem}

Next let us bring a
characterization for the case $n=3,$ which is not covered in above Theorem (see  \cite{HT}, Corollary 6.1).
\begin{proposition}[\cite{HV}, Prop. 3.3]\label{prp:n=3}
Let $\Xset$ be a $3$-poised set and ${\ell}$ be a $3$-node line. Then we have
that
\begin{equation} \label{2bin} |\Xset_{\ell}| = 3,\quad 1,\quad\hbox{or} \quad 0.
\end{equation}
Moreover, the following hold:
\vspace{-1.5mm} \begin{enumerate} \setlength{\itemsep}{0mm}
\item
$|\Xset_{\ell}| = 3$ if and only if there is a maximal line $\lambda_0$ such that $\lambda_0 \cap {\ell} \cap  \Xset =\emptyset.$ In this case
we have that $\Xset_{\ell}=\Xset\setminus (\lambda_0\cup \ell).$ Hence it is an $\ell$-proper $GC_{1}$ set.
\item
$|\Xset_{\ell}| = 1$ if and only if there are two maximal lines $\lambda', \lambda'',$ such that $\lambda' \cap \lambda'' \cap {\ell} \in \Xset.$
In this case we have that $\Xset_{\ell}=\Xset\setminus (\lambda'\cup \lambda''\cup \ell).$ Hence it is an $\ell$-proper $GC_{0}$ set.
\item $|\Xset_{\ell}| = 0$ if and only if there are exactly three maximal lines in $\Xset$ and they intersect $\ell$ at three distinct nodes.
\end{enumerate}
\end{proposition}
Let us mention that the statement \eqref{2bin} of Proposition \ref{prp:n=3} (without  the ``Moreover" part)  is valid for $3$-node lines in any  $n$-poised set (see \cite{HT}, Corollary 6.1). More precisely the following statement holds:
\begin{equation*} \hbox{If $\Xset$ is an $n$-poised set and $\ell$  is a $3$-node line then\ }  |\Xset_{\ell}| = 3,\quad 1,\quad\hbox{or} \quad 0.
\end{equation*}
Note that this statement contains all conclusions of Conjecture \ref{mainc} for the case of $3$-node lines, except the claim that the set $\Xset_\ell$ is an $\ell$-proper $GC_n$  subset in the cases $ |\Xset_{\ell}| = 3,1.$
And for this reason we cannot use it in proving Conjecture  \ref{mainc}.

Let us mention that, in view of the relation \eqref{maxaaa}, Conjecture \ref{mainc} is true if the line $\ell$ is a maximal line. Also Conjecture \ref{mainc} is true in the case when $GC_n$ set $\Xset$ is a Chung -Yao lattice.
Indeed, in this lattice the only used lines are the maximal lines. Also for any $k$-node line in $\Xset$ with $k\le n$ we have that $2k\le n+2,$ since through any node there pass two maximal lines. Thus for these non-used $k$-node lines we have $\sigma\le 1$ (see \cite{BH}).

The following proposition reveals a rich structure of the Carnicer-Gasca lattice.
\begin{proposition}[\cite{BH}, Prop. 3.8]\label{CGl}
Let $\Xset$ be a Carnicer-Gasca lattice of degree $n$ and
${\ell}$ be a $k$-node line. Then we have that
\begin{equation} \label{2bba}  \Xset_\ell \ \hbox{is an $\ell$-proper $GC_{s-2}$  subset  of $\Xset,$ hence}\  |\Xset_{\ell}| = \binom{s}{2},
\end{equation}
where $\sigma:=2k-n-1\le s \le k.$\\
Moreover, for any maximal line $\lambda$ we have:
$|\lambda\cap \Xset_{\ell}| = 0$ or $|\lambda\cap \Xset_{\ell}| = s-1.$\\
Furthermore, for each $n, k$ and $s$ with $\sigma\le s \le k,$  there is a Carnicer-Gasca lattice of degree $n$ and a $k$-node line $\ell$ such that \eqref{2bba} is satisfied.
\end{proposition}
Note that the phrase ``$\ell$-proper" is not present in the formulation of Proposition in \cite{BH} but it follows readily from the proof there.

Next consider the following statement of Carnicer and Gasca (see \cite{CG03}, Proposition 4.2):
\begin{equation} \label{2nodel}\hbox{If $\Xset$ is a $GC_n$ set and $\ell$  is a $2$-node line then\ }  |\Xset_{\ell}| = 1\  \ \hbox{or}\ \ 0.
\end{equation}

Let us adjoin this with the following statement:

If $\Xset$ is a $GC_n$ set, $\ell$  is a $2$-node line, and $|\Xset_\ell|=1,$ then $\Xset_\ell$ is an $\ell$-proper $GC_0$ subset, provided that $GM$ conjecture is true for all degrees up to $n.$

Indeed, suppose that $\Xset_\ell=\{A\}$ and $\ell$ passes through the nodes $B,C\in\Xset.$  The node $A$ uses a maximal $(n+1)$-node line in $\Xset$ which we denote by $\lambda_0.$ Next, $A$ uses a maximal $n$-node line in $\Xset\setminus \lambda_0$ which we denote by $\lambda_1.$ Continuing this way we find consecutively the lines $\lambda_2, \lambda_3,\ldots, \lambda_{n-1}$ and obtain that
$$\{A\}=\Xset\setminus (\lambda_0\cup \lambda_1\cup \cdots \cup \lambda_{n-1}).$$
To finish the proof it suffices to show that $\lambda_{n-1}=\ell$ and the remaining lines  $\lambda_i, i=0,\ldots,n-2$ are $\ell$-disjoint.
Indeed, the node $A$ uses $\ell$ and since it is a $2$-node line it may coincide only with the last maximal line $\lambda_{n-1}.$
Now, suppose conversely that a maximal line $\lambda_k,\ 0\le k\le n-2,$ intersects $\ell$ at a node, say $B.$ Then consider the polynomial of degree $n:$
$$p=\ell_{A,C}\prod_{i\in\{0,\ldots,n-1\}\setminus\{k\}}\lambda_i,$$ where $\ell_{A,C}$ is the line through $A$ and $C.$
Clearly $p$ passes through all the nodes of $\Xset$ which contradicts Proposition \ref{prp:poised}.

Now, in view of the statement \eqref{2nodel} and the adjoint statement, we conclude that Conjecture \ref{mainc} is true for the case of $2$-node lines in any $GC_n$ sets.

It is worth mentioning that the statement \eqref{2nodel} is true also for any $n$-poised set $\Xset$ (see \cite{BH}, relation (1.4), due to V. Bayramyan).

\subsection{Some preliminaries for the proof of Conjecture \ref{mainc} \label{s:conj}}

Here we prove two propositions. The following proposition shows that Conjecture \ref{mainc} is true if the node set $\Xset$ has exactly $3$ maximal lines.

\begin{proposition}\label{prop:3max}
Assume that GM Conjecture holds for all degrees up to $n-3$. Let $\Xset$ be a $GC_n$ set with exactly $3$ maximal lines, and
${\ell}$ be an  $m$-node line, $2\le m\le n+1.$ Then we have that
\begin{equation} \label{1bin3} \Xset_{\ell}=\emptyset,\ \hbox{or} \ \Xset \ \hbox{is a proper $GC_{m-2}$  subset of $\Xset,$ hence}\  |\Xset_{\ell}| = \binom{m}{2}.
\end{equation}
Moreover, if $\Xset_{\ell}\neq\emptyset$ and $m\le n$ then for a maximal line $\lambda_1$ of $\Xset$ we have:
\vspace{-1.5mm} \begin{enumerate} \setlength{\itemsep}{0mm}
\item
$\lambda_1 \cap {\ell} \notin \Xset$ and $|\lambda_1\cap \Xset_{\ell}| = 0.$

\item
$|\lambda\cap \Xset_{\ell}| = m-1\ $
for the remaining two maximal lines.
\end{enumerate}
Furthermore, if the line $\ell$ intersects each maximal line at a node then $\Xset_\ell=\emptyset.$
\end{proposition}
\begin{proof} According to Theorem \ref{th:CGo06} the set $\Xset$ is a generalized principal lattice of degree n with some three sets of $n+1$ lines \eqref{aaagpl}.  Then we obtain from \eqref{aaabbc} that the only used lines in $\Xset$ are the lines $\ell_{s}^r(\Xset),$ where  $0\le s< n,\ r=0,1,2.$ Therefore
the only used $m$-node lines are the lines $\ell_{n-m+1}^r(\Xset),\ r=0,1,2.$ Consider the line, say with $r=0,$ i.e., $\ell\equiv \ell_{n-m+1}^0(\Xset).$ It is used by all the nodes $x_{ijk}\in \Xset$ with $i>n-m+1,$ i.e., $i=n-m+2, n-m+3,\ldots,n.$ Thus, $\ell$ is used by exactly $\binom{m}{2}=(m-1)+(m-2)+\cdots+1$ nodes.
This implies also that $\Xset_\ell= \Xset\setminus (\ell_0\cup \ell_1\cup\cdots\cup \ell_{n-m+1}).$ Hence $\Xset$ is a proper $GC_{m-2}$  subset of $\Xset.$ The part ``Moreover" also follows readily from here. Now it remains to notice that the part ``Furthermore" is a straightforward consequence of the part ``Moreover".
\end{proof}

The following statement on the presence and usage of $(n-1)$-node lines in $GC_n$ sets with exactly $n-1$ maximal lines will be used in the sequel (cf. Proposition 4.2, \cite{HV}).
\begin{proposition}\label{prp:n-1}
Let $\Xset$ be a $GC_n$ set with exactly $n-1$ maximal lines and $\ell$ be an $(n-1)$-node line, where $n\ge 4.$ Assume also that
through each node of $\ell$ there passes exactly one maximal line.
Then we have that either
$n=4$  or $n=5.$ Moreover, in both these cases we have that $\Xset_\ell=\emptyset.$
\end{proposition}
\begin{proof}
Consider a $GC_n$ set with exactly $n-1$ maximal lines.  In this case we have the representation \eqref{01OOO}, i.e., $\Xset:=\Xset^\times\cup\Xset'''\cup \{O_1,O_2,O_3\},$ satisfying the conditions of  Proposition \ref{prp:n-1max}. Here $\Xset^\times$ is the set of all intersection nodes of the maximal lines, $\Xset'''$ is the set of the remaining ``free" nodes in the maximal lines, and $O_1,O_2,O_3$ are the three nodes outside the maximal lines. Let $\ell$ be an $(n-1)$-node line.

First notice that, according to the hypothesis of Proposition, all the nodes of the line $\ell$ are ``free" nodes. Therefore $\ell$ does not coincide with any $OO$ line, i.e., line passing through two outside nodes.

From Proposition \ref{prp:n-1max} we have that all the ``free" nodes except the three special nodes $A_1^1,A_2^2,A_3^3,$ which we will call here $(s)$ nodes, belong to the three $OO$ lines.
We have also, in view of Proposition \ref{prp:n-1max}, (iii), that the nodes $A_1^1,A_2^2,A_3^3$ are not collinear. Therefore there are three possible cases:
\vspace{-1.5mm} \begin{enumerate} \setlength{\itemsep}{0mm}
\item
$\ell$ does not pass through any $(s)$ node,
\item
$\ell$ passes through two $(s)$ nodes,
\item
$\ell$ passes through one $(s)$ node.
\end{enumerate}
In the first case $\ell$ may pass only through nodes lying in three $OO$ lines.
Then it may pass through at most three nodes, i.e., $n\le 4.$ Therefore, in view of the hypothesis $n\ge 4,$ we get $n=4.$ Since $\mu(\Xset)=3$ we get, in view of Proposition \ref{prop:3max}, part ``Furthermore", that $\Xset_\ell=\emptyset.$

Next, consider the case when  $\ell$ passes through two $(s)$ nodes. Then, according to Proposition \ref{prp:n-1max}, (iii), it passes through an outside $O$ node. Recall that this case is excluded since $\ell$ passes through ``free" nodes only.

Finally, consider the third case when $\ell$ passes through exactly one $(s)$ node.
Then it may pass through at most three other ``free" nodes lying in $OO$ lines. Therefore $\ell$ may pass through at most four nodes.

First suppose that $\ell$ passes through exactly $3$ nodes. Then again we obtain that
$n=4$ and  $\Xset_\ell=\emptyset.$

Next suppose that $\ell$ passes through exactly $4$ nodes.
Then we have that $n=5.$ Without loss of generality we may assume that the (s) node $\ell$ passes through is, say, $A_1^1.$
Next let us show first that $|\Xset_\ell|\le 1.$ Here we have exactly $4=n-1$ maximal lines.
Consider the maximal line $\lambda_4$ for which, in view of Proposition \ref{prp:n-1max}, the intersection with each $OO$ line is a node in $\Xset$ (see Fig. \ref{pic1}).
Denote $B:=\ell\cap \lambda_4.$ Assume that the node $B$ belongs to the line $\in \ell_i^{oo},\ 1\le i\le 3,$ i.e., the line passing through $\{O_1, O_2,O_3\}\setminus \{O_i\}$ ($i=2$ in Fig. \ref{pic1}).  According to the condition (ii) of Proposition \ref{prp:n-1max} we have that $\ell_i^{oo}\cap \lambda_i\notin \Xset$.
Denote $C:=\lambda_i\cap \lambda_4.$
\begin{figure}[ht] 
\centering
\includegraphics[scale=0.5]{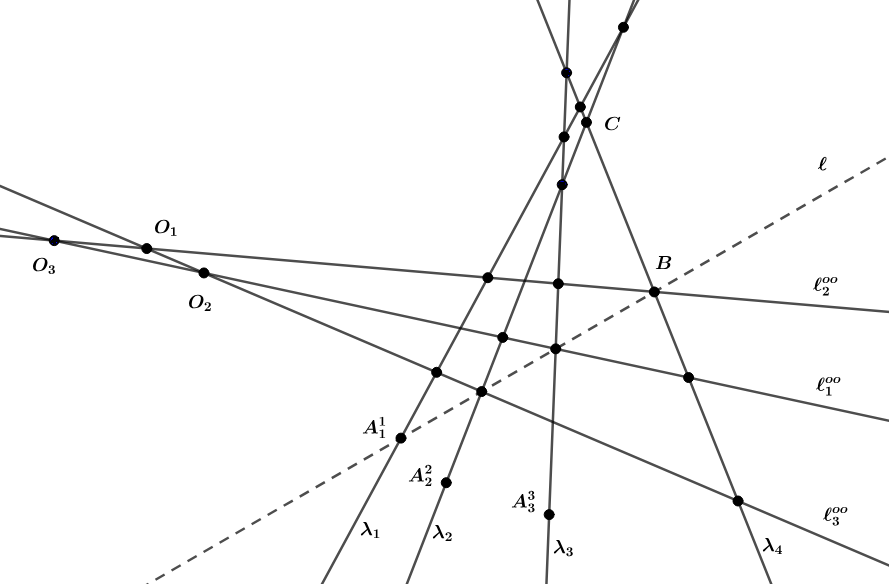}
\caption{The node set $\Xset$ for $n=5.$}\label{pic1}
\end{figure}
Now let us prove that $\Xset_\ell\subset \{C\}$ which implies $|\Xset_\ell|\le 1.$

Consider the $GC_4$ set $\Xset_i=\Xset\setminus \lambda_i.$ Here we have two maximal lines: $\lambda_4$ and $\ell_i^{oo},$
intersecting at the node $B\in \ell.$ Therefore we conclude from Lemma \ref{lem:CG2} that no node from these two maximal lines uses $\ell$ in $\Xset_2.$  Thus, in view of \eqref{rep}, no node from $\lambda_4,$ except possibly $C,$ uses $\ell$ in $\Xset.$ Now  consider the $GC_4$ set $\Xset_4=\Xset\setminus \lambda_4.$
Observe, on the basis of the characterization of Proposition \ref{prp:n-1max}, that $\Xset_4$ has exactly $3$ maximal lines.
On the other hand here the line $\ell$ intersects each maximal line at a node.
Therefore, in view of Proposition \ref{prop:3max}, part "Furthermore", we have that ${(\Xset_4)}_\ell=\emptyset.$
Hence, in view of \eqref{rep}, we conclude that $\Xset_\ell\subset \{C\}.$

Now, to complete the proof it suffices to show that the node $C$ does not use $\ell.$ Let us determine which lines uses $C.$ Since $C=\lambda_i\cap \lambda_4$ first of all it uses the two  maxmal lines $\{\lambda_1,\lambda_2,\lambda_3\}\setminus \{\lambda_i\}.$ It is easily seen that the next two lines $B$ uses are $OO$ lines: $\{\ell_1^{oo},\ell_2^{oo},\ell_3^{oo}\}\setminus\{\ell_i^{oo}\}.$ Now notice that, the two nodes, except $C$, which do not belong to the four used lines are $B$ and the (s) node $A_i^i.$ Hence the fifth line $B$ uses is the line passing through the latter two nodes. This line coincides with $\ell$ if and only if $i=1.$

In the final and most interesting part of the proof we will show that the case $i=1,$ when the node $B$ uses the line $\ell,$ i.e., the case when the (s) node $\ell$ passes is $A_1^1$ and $B\in  \ell_1^{oo},$ where $B=\ell\cap\lambda_4,$ is impossible.
\begin{figure}[ht] 
\centering
\includegraphics[scale=0.6]{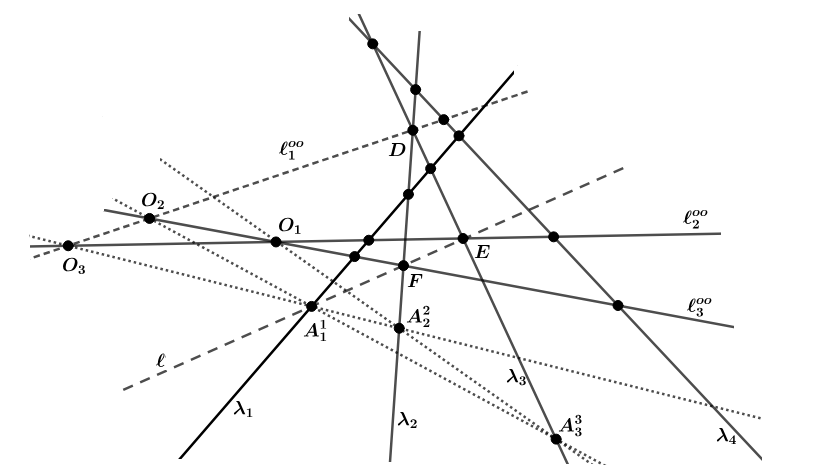}
\caption{The case $i=1.$}\label{pic12}
\end{figure}
More precicely, we will do the following (see Fig. \ref{pic12}). By assuming that
\begin{enumerate}
\item
the maximal lines $\lambda_i, i=1,\ldots,4,$ are given, the three outside nodes $O_1,O_2,O_3$ are given,
\item
the two $OO$ lines $\ell_2^{oo},\ell_3^{oo}$ are intersecting the three maximal lines at $6$ distinct points, and
\item
the three conditions in Proposition \ref{prp:n-1max}, part ``Moreover" are satisfied,   i.e., the line through the two special nodes $\{A_1^1,A_2^2,A_3^3\}\setminus A_i^i\}$  passes through the outside node $O_i$ for each $i=1,2,3,$
\end{enumerate}
we will prove that the third $OO$ line $\ell_1^{oo}$  passes necessarily through the the node $D,$ i.e., intersection node of the two maximal lines $\lambda_2$ and $\lambda_3.$ Since this makes the number of all nodes of the set $\Xset$ only $19$ instead of $21,$ we may conclude that the abovementioned case is impossible.

For this end, to simplify the  Fig. \ref{pic12} let us delete from it the maximal lines $\lambda_1$ and $\lambda_4$ to obtain the following Fig. \ref{pic15}.
\begin{figure}[ht] 
\centering
\includegraphics[scale=0.6]{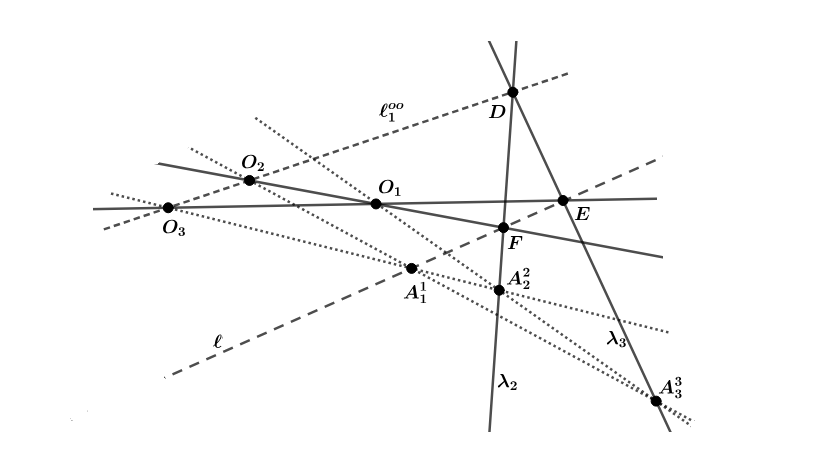}
\caption{The $(9_3)$ configuration with $3$ triangles.}\label{pic15}
\end{figure}
Let us now apply the well-known Pappus hexagon theorem for the pair of triple collinear nodes here
$$A_1^1,\ \ E,\ \ F;$$   $$O_1,\ \ A_2^2,\ \ A_3^3.$$
Now observe that $$\ell(A_1^1,A_2^2) \cap \ell(E,O_1)=O_3,\ \ell(E,A_3^3) \cap \ell(F,A_2^2)=D,\ \ell(A_1^1,A_3^3) \cap \ell(F,O_1)=O_2,$$
where $\ell(A,B)$ denotes the line passing through the points $A$ and $B.$ 
Thus, according to the Pappus theorem we get that the triple of nodes $D,O_2,O_3$ is collinear.
\end{proof}

Notice that in Fig. \ref{pic15} we have a $(9_3)$ configuration of 9 lines and 9 points (three triangles) that occurs with each line meeting 3 of the points and each point meeting 3 lines (see \cite{HCV}, Chapter 3, Section 17).
\begin{figure}[ht] 
\centering
\includegraphics[scale=0.45]{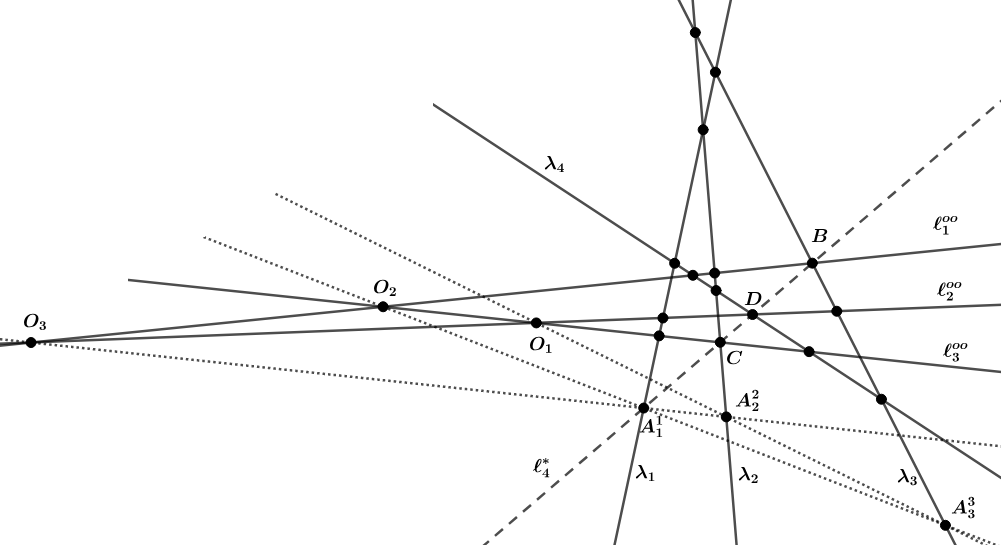}
\caption{A non-used $4$-node line $\ell_4^*$ in a $GC_5$ set $\Xset^*$}\label{pic11}
\end{figure}
\begin{remark} \label{rm} Let us show that  the case of non-used $4$-node line in  Figure \ref{pic1} is possible nevertheless. The problem with this is that we have to confirm that the three conditions in Proposition \ref{prp:n-1max}, are satisfied. More precicely:
\begin{enumerate}
\item
$\ell_i^{oo}\cap \lambda_i=\emptyset$ for each $i=1,2,3;$
\item 
The line through the two special nodes $\{A_1^1,A_2^2,A_3^3\}\setminus A_i^i\}$  passes through the outside node $O_i$ for each $i=1,2,3.$ 
\end{enumerate}
Let us outline how one can get a desired figure (see Fig. \ref{pic11}). Let us start the  figure with the three maximal lines (non-concurrent) $\lambda_1, \lambda_2,\lambda_3.$ Then we choose two $OO$ lines lines $\ell_1^{oo}, \ell_3^{oo}$ through the outside node $O_2,$ which intersect the three maximal lines at $6$ distinct points. Next we get the line $\ell_4^*$ which passes through the points $B$ and $C,$ where $B=\ell_1^{oo}\cap \lambda_3$ and $C=\ell_3^{oo}\cap \lambda_2.$
Now we find the point $A_1^1:=\ell_4^*\cap \lambda_1.$ By intersecting the line through $O_2$ and $A_1^1$ with the maximal line $\lambda_3$ we get the node $A_3^3.$ Then we choose a ``free node" $A_2^2$ on the maximal line $\lambda_2.$ This enables to determine the two remaining outside nodes: $O_1, O_3.$ Namely we have that $O_1$ is the intersection point of  
the line through $A_2^2$ and $A_3^3$ with the line $\ell_3^{oo}$ and $O_3$ is the intersection point of  
the line through $A_2^2$ and $A_1^1$ with the line $\ell_1^{oo}.$  Thus we get the line $\ell_2^{oo}$ which passes through the outside nodes $O_1, O_3.$
Next we get the points of intersection of the line  $\ell_2^{oo}$ with the three maximal lines  as well as the point of intersection with the line $\ell_4^*$ denoted by $D.$
Now, we choose the maximal line $\lambda_4$ passing through $D$ and intersecting the remaining maximal lines and  the three $OO$ lines at $6$ new distinct nodes.
Finally, all the specified intersection points in  Fig. \ref{pic11} we declare as the nodes of  $\Xset^*.$\end{remark}

\section{Proof of Conjecture \ref{mainc} \label{ss:conj}}

Let us start the proof with a list of the major cases in which Conjecture \ref{mainc} is  true.

\vspace{1.5mm}

\noindent \emph{Step 1.} The Conjecture \ref{mainc} is true in the following cases:
\vspace{-1.5mm} \begin{enumerate} \setlength{\itemsep}{0mm}
\item
The line $\ell$ is a maximal line.

Indeed, as we have mentioned already, in this case we have $\Xset_\ell=\Xset\setminus \ell$ and all the conclusions of Conjecture can be readily verified.

\item
The line $\ell$ is an $n$-node line, $n\in \mathbb N$.

In this case Conjecture \ref{mainc} is valid by virtue of Theorem \ref{th:corrected} (for $n\in \mathbb N\setminus \{3\}$) and Proposition \ref{prp:n=3} (for $n=3$).

\item
The line $\ell$ is a $2$-node line.

In this case Conjecture \ref{mainc} follows from the statement \eqref{2nodel} and the adjoint statement next to it.
\end{enumerate}

Now, let us prove Conjecture by complete induction on $n$ - the degree of the node set $\Xset.$
Obviously Conjecture is true in the cases $n=1,2.$
Note that this follows also from Step 1 (i) and (ii).

Assume that Conjecture is true for any node set of degree not exceeding $n-1.$ Then let us prove that it is true for the node set $\Xset$ of degree $n.$
Suppose that we have a $k$-node line $\ell.$
\vspace{1.5mm}

\noindent \emph{Step 2:} Suppose additionally that there is an $\ell$-disjoint maximal line $\lambda.$
Then we get from Lemma \ref{lem:CG1} that
\begin{equation}\label{abb}\Xset_{\ell} = {(\Xset \setminus \lambda)}_{\ell}.
\end{equation}
Therefore by using the induction hypothesis for the $GC_{n-1}$ set
$\Xset':=\Xset \setminus \lambda$ we get the relation \eqref{1bbaa}, where $\sigma':=\sigma(\Xset',\ell)\le s \le k$ and $\sigma'=2k-(n-1)-1=2k-n=\sigma+1.$
Here we use also Proposition \ref{proper} in checking that $\Xset_\ell$ is an $\ell$-proper subset of $\Xset.$

Now let us verify the part "Moreover". Suppose that $\sigma=2k-n-1\ge 3,$ i.e. $2k\ge n+4,$ and $\mu(\Xset)>3.$ For the line $\ell$ in the $GC_{n-1}$ set $\Xset_0$  we have $\sigma'=\sigma+1\ge 4.$ Thus if $\mu(\Xset')>3$ then, by the induction hypothesis, we have that $(\Xset')_\ell\neq\emptyset.$ Therefore we get, in view of \eqref{abb}, that $\Xset_\ell\neq\emptyset.$
It remains to consider the case $\mu(\Xset')=3.$
In this case, in view of Proposition \ref{prp:CG-1}, we have that $\mu(\Xset)=4,$ which, in view of Theorem \ref{th:CGo10}, implies that $4\in \{n-1,n,n+1,n+2\},$ i.e., $2\le n\le 5.$

The case $n=2$ was verified already.
Now, since $2k\ge n+4$ we deduce that either
$k\ge 4$ if $n=3, 4,$ or $k\ge 5$ if $n=5.$
All these cases follow from Step 1 (i) or (ii).

The part "Furthermore" follows readily from the relation \eqref{abb}.

\vspace{1.5mm}

\noindent \emph{Step 3:} Suppose additionally that there is a pair of $\ell$-adjacent maximal lines $\lambda', \lambda''.$
Then we get from Lemma \ref{lem:CG2} that
\begin{equation}\label{bbac}
\Xset_{\ell} = {(\Xset \setminus (\lambda' \cup \lambda''))}_{\ell}.
\end{equation}

Therefore by using the induction hypothesis for the $GC_{n-2}$ set
$\Xset'':=\Xset \setminus (\lambda' \cup \lambda''))$ we get the relation \eqref{1bbaa},  where $\sigma'':=\sigma(\Xset'',\ell)\le s \le k-1$ and $\sigma''=2(k-1)-(n-2)-1=\sigma.$
Here we use also Proposition \ref{proper} to check that $\Xset_\ell$ is an $\ell$-proper subset of $\Xset.$

Let us verify the part "Moreover". Suppose that $\sigma=2k-n-1\ge 3$ and $\mu(\Xset)>3.$ The line $\ell$ is $(k-1)$-node line in the $GC_{n-2}$ set $\Xset''$ and we have that $\sigma''=\sigma\ge 3.$ Thus if $\mu(\Xset'')>3$ then, by the induction hypothesis, we have that $(\Xset'')_\ell\neq\emptyset$ and therefore we get, in view of \eqref{bbac},  that $\Xset_\ell\neq\emptyset.$
It remains to consider the case $\mu(\Xset'')=3.$
Then, in view of Proposition \ref{prp:CG-1}, we have that $\mu(\Xset)=4$ or $5,$ which, in view of Theorem \ref{th:CGo10}, implies that $4\ \hbox{or}\ 5\in \{n-1,n,n+1,n+2,\}$ i.e., $2\le n\le 6.$

The cases  $2\le n\le 5$ were considered in the previous step.
Thus suppose that $n=6.$
Then, since $2k\ge n+4,$ we deduce that
$k\ge 5.$ In view of Step 1, (ii), we may suppose that $k=5.$

Now the set $\Xset''$ is a $GC_4$  and the line $\ell$ is a $4$-node line there. Thus, in view of Step 1 (ii) we have that
$(\Xset'')_\ell\neq \emptyset.$ Therefore we get, in view of \eqref{bbac}, that
 $\Xset_\ell\neq \emptyset.$

The part "Furthermore" follows readily from the relation \eqref{bbac}.

\vspace{1.5mm}

\noindent \emph{Step 4.} Now consider any $k$-node line $\ell$  in a $GC_n$ set $\Xset.$ In view of Step 1 (iii) we may assume that $k\ge 3.$ In view of Theorem \ref{th:CGo10} and Proposition \ref{prop:3max} we may assume also that $\mu(\Xset)\ge n-1.$

Next suppose that $k\le n-2.$ Since then $\mu(\Xset)>k$ we necessarily
have either the situation of Step 2 or Step 3.

Thus we may assume that $k\ge n-1.$ Then, in view of Step 1 (i) and (ii), it remains to consider the case $k =n-1,$ i.e., $\ell$ is an $(n-1)$-node line. Again if $\mu(\Xset)\ge n$ then we necessarily
have either the situation of Step 2 or Step 3. Therefore we may assume also that $\mu(\Xset)=n-1.$ By the same argument we may assume that each of the $n-1$ nodes of the line $\ell$  is an intersection node with one of the $n-1$ maximal lines.

Therefore the conditions of Proposition \ref{prp:n-1} are satisfied and we arrive to the two cases: $n=4, k=3,  \sigma=1$ or $n=5, k=4,  \sigma=2.$
In both cases we have that $\Xset_\ell=\emptyset.$ Thus in this case Conjecture is true.
\hfill $\square$

\subsection{The \label{ss:count} characterization of the case $\sigma=2,\ \mu(\Xset) >3$}

Here we bring, for each $n$ and $k$ with $\sigma=2k-n-1 =2,$
two constructions of $GC_n$ set and a nonused $k$-node line there. At the end (see forthcoming Proposition \ref{charact}) we prove that these are the only constructions with the mentioned property.

Let us start with a counterexample in the case $n=k=3$ (see \cite{HV}, Section 3.1), i.e., with  a $GC_3$ set $\Yset^*$ and a $3$-node line $\ell_3^*$ there which is not used.

\begin{figure}[ht] 
\centering
\includegraphics[scale=0.6]{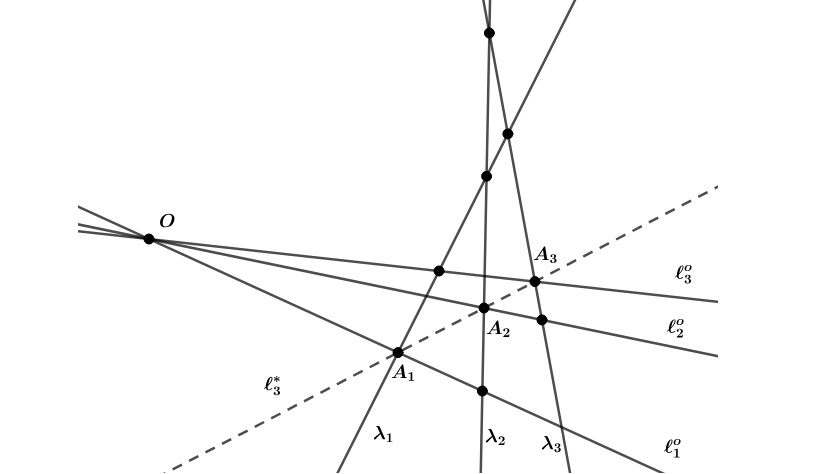}
\caption{A non-used $3$-node line $\ell_3^*$ in a $GC_3$ set $\Yset^*$}\label{pic2}
\end{figure}

Consider a $GC_3$ set $\Yset^*$ of $10$ nodes with exactly three maximal lines: $\lambda_1,\lambda_2,\lambda_3$ (see Fig. \ref{pic2}). This set has construction \eqref{01O} and satisfies the conditions listed in Proposition \ref{prp:nmax}. Now observe that the $3$-node line $\ell_3^*$ here intersects all the three maximal lines at nodes. Therefore, in view of Proposition \ref{prp:n=3}, (iii), the line $\ell_3^*$ cannot be used by any node in $\Yset^*,$ i.e., $$(\Yset^*)_{\ell_3^*}=\emptyset.$$

Let us outline how one can get Fig. \ref{pic2}. We start the mentioned figure with the three lines $\ell_1^o, \ell_2^o,\ell_3^o$ through $O,$ i.e., the outside node. Then we choose the maximal lines $\lambda_1,\lambda_2,$ intersecting $\ell_1^o, \ell_2^o$ at $4$ distinct points. Let $A_i$ be the intersection point $\lambda_i\cap\ell_i^o,\ i=1,2.$  We chose the points $A_1$ and $A_2$ such that the line through them: $\ell_3^*$  intersects the line $\ell_3^o$ at a point $A_3.$ Next we choose a third maximal line $\lambda_3$ passing through $A_3.$ Let us mention that we chose the maximal lines such that they are not concurrent and intersect the three lines through $O$ at nine distinct points. Finally, all the specified intersection points in  Fig. \ref{pic2}  we declare as the nodes of  $\Yset^*.$

In the general case of $\sigma=2k-n-1=2$ we set $k=m+3$ and obtain $n=2m+3,$ where $m=0,1,2,\ldots$. Let us describe how the previous $GC_3$ node set 
$\Yset^*$ together with the $3$-node line $\ell_3^*$ can be modified to $GC_n$ node set $\bar{\Xset}^*$ with a $k$-node line $\bar{\ell}_3^*$
in a way to fit the general case. That is to have that  $(\bar{\Yset^*})_{\bar{\ell_3}}=\emptyset.$
\begin{figure}[ht] 
\centering
\includegraphics[scale=0.5]{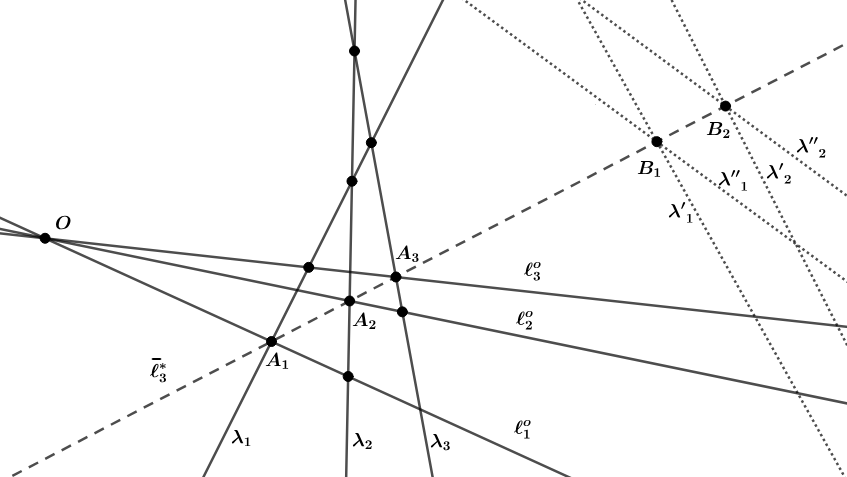}
\caption{A non-used $k$-node line $\bar{\ell}_3^*$ in a $GC_n$ set $\bar{\Yset}^*$ with $\sigma=2\ (m=2)$\label{pic3}}
\end{figure}

For this end we just leave the line $\ell_3^*$ unchanged, i.e., $\bar{\ell}_3^*\equiv\ell_3^*$ and extend the set $\Yset^*$ to a $GC_n$ set $\bar{\Yset}^*$ in the following way. We fix $m$ points: $B_i, i=1,\ldots,m,$ in $\ell_3^*$ different from $A_1,A_2,A_3$ (see Fig. \ref{pic3}). Then we add $m$ pairs of (maximal) 
lines $\lambda_i',\lambda_i'',  i=1,\ldots,m,$ intersecting at these $m$ points respectively: $$\lambda_i'\cap \lambda_i'' = B_i,\ i=1,\ldots,m.$$

We assume that the following condition is satisfied:
\begin{enumerate}
\item
The $2m$ lines $\lambda_i',\lambda_i'', i=1,\ldots,m,$ together with $\lambda_1,\lambda_2,\lambda_3$ are in general position, i.e., no two lines are parallel and no three lines are
concurrent;
\item
The mentioned $2m+3$ lines intersect the lines $\ell_1^o, \ell_2^o,\ell_3^o$ at distinct $3(2m+3)$ points.
\end{enumerate}
Now all the points of intersections of the $2m+3$ lines $\lambda_i',\lambda_i'', i=1,\ldots,m,$ together with $\lambda_1,\lambda_2,\lambda_3$ are declared as the nodes of the set $\bar{\Yset}^*.$  Next for each of the lines $\lambda_i',\lambda_i'', i=1,\ldots,m,$ also two from the three intersection points with the lines $\ell_1,\ell_2,\ell_3,$ are declared as (``free") nodes.
After this the lines  $\lambda_1,\lambda_2,\lambda_3$ and the lines $\lambda_i',\lambda_i'', i=1,\ldots,m,$ become  $(2m+4)$-node lines, i.e., maximal lines.

Now one can verify readily that the set  $\bar{\Yset}^*$ is a $GC_n$ set since it satisfies the construction \eqref{01O} with $n=2m+3$ maximal lines: $\lambda_i',\lambda_i'', i=1,\ldots,m,$ together with $\lambda_1,\lambda_2,\lambda_3,$ and satisfies the conditions listed in Proposition \ref{prp:nmax}.

Finally, in view of Lemma \ref{lem:CG2} and the relation \eqref{bbb}, applied $m$ times with respect to the pairs $\lambda_i',\lambda_i'', i=1,\ldots,m,$ gives:
${(\bar{\Yset}^*)}_{\bar{\ell}_3^*}=(\Yset^*)_{\ell_3^*}=\emptyset.$

Let us call the set $\bar{\Yset}^*$ an $m$-modification of the set $\Yset^*.$ In the same way we could define $\bar{\Xset}^*$ as an $m$-modification of the set
$\Xset^*$ from Fig. \ref{pic1}, with the $4$-node non-used line $\ell_4^*$ (see Remark \ref{rm}).  The only differences from the previous case here are:
\begin{enumerate}
\item
 Now $k=m+4,\ n=2m+5, m=0,1,2,\ldots$ (again $\sigma=2$);
\item
We have $2n+4$ maximal lines: $\lambda_i',\lambda_i'', i=1,\ldots,m,$ and the lines $\lambda_1,\lambda_2,\lambda_3,\lambda_4;$
\item
Instead of the lines  $\ell_1^o, \ell_2^o,\ell_3^o$  we have the lines $\ell_{1}^{oo}, \ell_{2}^{oo}, \ell_{3}^{oo};$
\item
For each of the lines $\lambda_i',\lambda_i'', i=1,\ldots,m,$ all three intersection points with the lines $\ell_1^{oo},\ell_2^{oo},\ell_3^{oo},$ are declared as (``free") nodes.
\end{enumerate}
Now one can verify readily that the set  $\bar{\Xset}^*$ is a $GC_n$ set since it satisfies the construction \eqref{01OOO} with $n=2m+4$ maximal lines,
 and satisfies the conditions listed in Proposition \ref{prp:n-1max}.

Thus we obtain another construction of non-used $k$-node lines in $GC_n$ sets, with $\sigma=2,$ where $k=m+4, n=2m+5.$

At the end let us prove the following
\begin{proposition}\label{prp:sigma2}\label{charact}
Let $\Xset$ be a $GC_n$ set and $\ell$ be a $k$-node line with $\sigma:=2k-n-1=2$ and $\mu(\Xset)>3.$
Suppose that the line $\ell$ is a non-used line. Then we have that either $\Xset=\bar{\Xset}^*, \ \ell={\bar{\ell}_4^*},$ or $\Xset=\bar{\Yset}^*, \ \ell={\bar{\ell}_3^*}.$
\end{proposition}
\begin{proof} Notice that $n$ is an odd number and $n\ge 3.$ In the case $n=3$ Proposition \ref{prp:sigma2} follows from Proposition \ref{prp:n=3}.

Thus suppose that $n\ge 5.$ Since $\mu(\Xset)>3,$ we get, in view of Theorem \ref{th:CGo10}, that
\begin{equation}\label{n-1}\mu(\Xset)\ge n-1.
\end{equation}

Next, let us prove that there is no $\ell$-disjoint maximal line $\lambda$ in $\Xset.$

Suppose conversely that $\lambda$ is a maximal line with $\lambda\cap \ell\notin\Xset.$
Denote by $\Xset':=\Xset\setminus \lambda.$ Since $\Xset_\ell=\emptyset$ therefore, by virtue of the relation \eqref{rep}, we obtain that $(\Xset')_\ell=\emptyset.$
Then we have that
$\sigma':=\sigma(\Xset',\ell)=2k-(n-1)-1=3.$
By taking into account the latter two facts, i.e., $(\Xset')_\ell=\emptyset$ and $\sigma'=3,$  we conclude from Conjecture \ref{mainc}, part ``Moreover", that $\mu(\Xset')=3.$
Next, by using \eqref{n-1} and Proposition \ref{prp:CG-1}, we obtain that that $\mu(\Xset)=4.$ By applying again \eqref{n-1} we get that $4\ge n-1,$
i.e., $n\le 5.$ Therefore we arrive to the case: $n=5.$ Since $\sigma=2$ we conclude that $k=4.$
Then observe that the line $\ell$ is $4$-node line in the $GC_4$ set $\Xset'$. By using  Theorem \ref{th:corrected} we get that $(\Xset')_\ell\ne\emptyset,$ which is a contradiction.

Now let us prove Proposition in the case $n=5.$ As we mentioned above then $k=4.$ We have that there is no $\ell$-disjoint maximal line. Suppose also that there is no  pair of $\ell$-adjacent maximal lines. Then, in view of \eqref{n-1}, we readily get that $\mu(\Xset)= 4$ and through each of the four nodes of the line $\ell$ there passes a maximal line. Therefore,  Proposition \ref{prp:n-1max} yields that $\Xset$ coincides with $\Xset^*$ (or, in other words, $\Xset$ is a $0$-modification of $\Xset^*$) and $\ell$ coincides with $\ell_4^*.$

Next suppose that there is a  pair of $\ell$-adjacent maximal lines: $\lambda',\lambda''.$ Denote by $\Xset'':=\Xset\setminus (\lambda'\cup\lambda'').$ Then we have that
$\ell$ is a $3$-node non-used line in $\Xset''.$ Thus we conclude readily that $\Xset$ coincides with $\bar{\Yset}^*$ and $\ell$ coincides with ${\bar\ell}_3^*$ (with $m=1$).

Now let us continue by using induction on $n.$ Assume that Proposition is valid for all degrees up to $n-1.$ Let us prove it in the case of the degree $n.$ We may suppose that $n\ge 7.$
It suffices to prove that there is a  pair of $\ell$-adjacent maximal lines: $\lambda',\lambda''.$ Indeed, in this case we can complete the proof just as in the case $n=5.$

Suppose by way of contradiction that there is no pair of $\ell$-adjacent maximal lines. Also we have that there is no $\ell$-disjoint maximal line. Therefore we have that
$\mu(\Xset) \le k.$ Now, by using \eqref{n-1}, we get that $k\ge n-1.$ Therefore $2=\sigma=2k-n-1\ge 2n-2-n-1=n-3.$
This implies that $n-3\le 2,$ i.e., $n\le 5,$ which is a contradiction.
\end{proof}

\vspace{3mm}


\noindent \emph{Hakop Hakopian} \vspace{2mm}

\noindent{Department of Informatics and Applied Mathematics\\
Yerevan State University\\
A. Manukyan St. 1\\
0025 Yerevan, Armenia}

\vspace{1mm}
\noindent E-mail: $<$hakop@ysu.am$>$

\vspace{5mm}

\noindent \emph{Vahagn Vardanyan} \vspace{2mm}

\noindent{Institute of Mathematics\\
National Academy of Sciences of Armenia\\
24/5 Marshal Baghramian Ave. \\
0019 Yerevan, Armenia}

\vspace{1mm}

\noindent E-mail:$<$vahagn.vardanyan@gmail.com$>$
\end{document}